\documentclass[a4paper,11.5pt,reqno]{amsart}

\usepackage{graphicx}
\usepackage[utf8]{inputenc}
\usepackage{amsmath,amssymb,amsthm}
\usepackage{xcolor}
\usepackage{hyperref}

\usepackage{lineno}
\modulolinenumbers[3]

\theoremstyle{plain}
\newtheorem{theorem}{Theorem}[section]

\theoremstyle{definition}
\newtheorem{definition}{Definition}

\theoremstyle{remark}
\newtheorem{remark}{Remark}

\newcommand{\R}{\mathbb{R}}
\newcommand{\Z}{\mathbb{Z}}

\numberwithin{equation}{section}

\hypersetup{colorlinks = true, linkcolor = blue, urlcolor = blue, citecolor = blue}

\begin{document}

\title[Families of lattices] 
      {Families of lattices with an unbounded number of unit vectors}

\author[H. Ruhland]{Helmut Ruhland}
\address{Santa F\'{e}, La Habana, Cuba}
\email{helmut.ruhland50@web.de}



\begin{abstract}
3 families of 4-dimensional lattices $L_k, M_k, M_k / 2 \subset \mathbb{R}^2$ are defined. Each lattice is defined by 2 quadratic extensions and has a \emph{finite} number of unit vectors, but the number of unit vectors in each of the 3 familes is \emph{unbounded}. $L_3$ is the Moser lattice.
\end{abstract}

\date{\today}

\maketitle

\section{Introduction}

In \cite{Rad} D. Radchenko gave an affirmative answer to the question posed by \cite[p.~186]{BraMoPa}: \emph{Does there exist a finitely generated additive subgroup $\mathcal{A} \subset \R^2$ such that there are infinitely many elements of $\mathcal{A}$ lying on the unit circle?} \\
In the same article it is also shown, a lattice with an infinite number of unit vectors has dimension $\ge 4$. \\

In this article we define three families $L_k, M_k, M_k / 2 \subset \R^2$ of 4-dimensional lattices obtained by $2$ quadratic extensions. The Moser lattice is $L_3$. Though none of the family members has an infinite number of unit vectors as in the lattices above, the number of unit vectors in the family is \emph{unbounded}.

\section{A family of 4-dimensional lattices with hexagonal sub-lattices \\ an an unbounded number of unit vectors}

\begin{definition}
\label{Def_L_k}
Let $\omega_k = \frac{(2 \, k - 1) + i \, \sqrt{4 \, k - 1}} {2 \, k} =
                e ^ {i \, \arccos (1 - 1 / (2 \, k))}$, 
$\omega_1 = \frac{1 + i \, \sqrt{3}} {2}$ is a $6^{th}$ root of unity, i.e. hexagonal sub-lattices. Define the following 4-dimensional lattice depending on the integer parameter $k$:
\begin{equation}
   \L_k = \{a \cdot 1 + b \cdot \omega_1 + c \cdot \omega_k + d \cdot \omega_1 \, \omega_k
                    \, \vert \, a, b, c, d \in \Z  \}   
   \label{L_k}
\end{equation}
The values $k = 1, 7, 19, 37, 61, 91, \dots$ have to be excluded. These are the values $(3 \, s^2 - 1) / 4$ for odd $s$. They have to be excluded, because in these cases  $4 \, k - 1 = 3 \, s^2$ and the lattice is only 2-dimensional. $L_3$ is the Moser lattice with $18$ unit vectors, see e.g. \cite{EnHaSuVaZs}.
\end{definition}

\begin{theorem}
Let $H (a, b) = a^2 + a b + b^2$ be the squared length of the vector $\vert \, a \cdot 1 + b \cdot \omega_1 \, \vert$ in a hexagonal lattice. Then for all integers $a, b$ with and $H (a, b) = k$, $(a, b, -a, -b)$ is a unit vector in $L_k$. The other $12$ unit vectors $(a, b, 0, 0)$ and $(0, 0, c, d)$ are defined by $H (a, b) = H (c, d) = 1$. \\
The number of solutions $a, b$ of $H (a, b) = k$, i.e. the number of unit vectors $- \, 12$ in the family $L_k$, is unbounded.
\end{theorem}

\begin{proof}
$(a, b, -a, -b) = (a \cdot 1 + b \cdot \omega_1) \, (1 - \omega_k)$, because $\vert \, 1 - \omega_k \, \vert = 1 / \sqrt{k}$, \\ $\vert \, (a, b, -a, -b) \, \vert = 1$ and is a unit vector. \\
Unboundedness: see e.g. \cite{Hir} formula (3), the number of solutions of $H (a, b) = k$ is denoted by $u (k)$ and $u (k) = 6 \, \left(d_{3,1} (k) - d_{3,2} (k)\right)$. Let $k = \prod_{i = 1}^{m} p_i$ be the product of $m$ primes with $p_i \equiv 1 \mod 3$. Then all $2^m$ divisors of $k$ are $\equiv 1 \mod 3$ and $u (k) = 6 \, 2^m$ is unbounded.
\end{proof}

More generally, with the prime decomposition of $k$, the primes $p_i \equiv 1 \mod 3, q_i \equiv 2 \mod 3$ and $u (k) \ne 0$:
\begin{equation}
   k = 3^t \, \prod_{i = 1}^{m} p_i^{e_i} \, \prod_{i = 1}^{n} q_i^{2f_i} \qquad
   u (k) = 6 \, \prod_{i = 1}^{m} (e_i + 1)
   \label{Gen_u_k}
\end{equation}
Note that $u (k) = 0$ when one of the exponents of the primes $q_i$ is odd. $u (k)$ is closely related to the factorization of $k$, not only in $\Z$, but in the ring of Eisenstein integers $\Z[\omega_1] = \Z[(1 + i \, \sqrt{3}) / 2]$. \\

One can find in N.J.A. Sloane's OEIS \href{https://oeis.org/A004016} {A004016} and \href{https://oeis.org/A004016/b004016.txt} {Table for n = 0..10000} the number of vectors in a shell of the hexagonal lattice.  
The lowest $k$ for a given number of solutions (divided by $6$) of $H (a, b) = k$ can be found in  \href{https://oeis.org/A343771} {A343771} and \href{https://oeis.org/A343771/b343771.txt} {Table for n = 0..1000}. The table \ref{Tab_L_k_lowest} is created using OEIS A343771.\\

\begin{table}[h]
  \caption{Lowest integer $k$ for a given number of unit vectors, obtained from
           \href{https://oeis.org/A343771/b343771.txt} {A343771 Table}, \\
           \# OEIS table entry = (\# unit vectors - 12) / 6}
  \begin{tabular}{c || c | c | c | c | c | c | c | c | c | c |}
  \hline
  \# unit vectors   & 18 & 24$^*$ & 30 &  36$^*$ &   42 &  48 &     56 &   60 & \dots &          300 \\
  \hline
  $k$               &  3 & 13     & 49 & 133     & 2\,401 & 637 & 117\,649 & 1\,729 & \dots & 13\,882\,141 \\
  \hline
  \end{tabular}
  \label{Tab_L_k_lowest}
\end{table}
{\footnotesize $ ^*$ In the OEIS table the lowest $k$ for the \# unit vectors  $= 24, 36$, i.e. for the entries $2, 4$ in OEIS, are $7, 91$. But these $L_k$ build only 2-dimensional lattices. So the next lowest $k$ with the same shell size is taken, that builds a 4-dimensional lattice.}

\subsection{Half integer subscripts for $L_k$}

Let $k$ be the half of an odd integer, i.e. $k = (2 m + 1) / 2$. In this case the conditions for non-trival init vectors for the integers $a, b$ reads as $H (a, b) = a^2 + a b + b^2 = (2 m + 1) /2$. There are no solutions. If we take the half of the four bases of $L_k$, naming the new lattice as $L_k / 2$ the
condition is $H (a / 2, b / 2) = (2 m + 1) / 2$ or $H (a, b) = 2 \, (2 m + 1)$. Because of the odd power of the factor $2$, see \ref{Gen_u_k}, $2$ is a prime $q_i \equiv 2 \mod 3$,  there exist no solutions. So these $L_k$ have only the trivial $12$ unit vectors and are not regarded as members of a own family.

\section{Another family $M_k$ of 4-dimensional lattices with square  \\ sub-lattices and an unbounded number of unit vectors}

\begin{definition}
Replace the $\omega_1$ in all $L_k$, see definition \ref{Def_L_k}, by $\omega_{1/2} = i$, i.e. square sub-lattices. Then we get a new family of lattices:
\begin{equation}
   M_k = \{a \cdot 1 + b \cdot \omega_{1/2} + c \cdot \omega_k + d \cdot \omega_{1/2} \, \omega_k
                    \, \vert \, a, b, c, d \in \Z  \}   
   \label{M_k}
\end{equation}
No integer values of $k$ have to be excluded. In the cases  $4 \, k - 1 = s^2$ the lattice wold be only 2-dimensional, but $-1$ is no quadratic residue modulo $4$.
\end{definition}

\begin{theorem}
Let $S (a, b) = a^2 + b^2$ be the squared length of the vector $\vert \, a \cdot 1 + b \cdot \omega_{1/2} \, \vert$ in a square lattice. Then for all integers $a, b$ with and $S (a, b) = k$, $(a, b, -a, -b)$ is a unit vector in $M_k$. The other $8$ unit vectors $(a, b, 0, 0)$ and $(0, 0, c, d)$ are defined by $S (a, b) = S (c, d) = 1$. \\
The number of solutions $a, b$ of $S (a, b) = k$, i.e. the number of unit vectors $- \, 8$ in the family $M_k$, is unbounded.
\end{theorem}

\begin{proof}
$(a, b, -a, -b) = (a \cdot 1 + b \cdot \omega_{1/2}) \, (1 - \omega_k)$, because $\vert \, 1 - \omega_k \, \vert = 1 / \sqrt{k}$, \\ $\vert \, (a, b, -a, -b) \, \vert = 1$ and is a unit vector. \\
Unboundedness: see e.g. \cite{Hir2004} theorem 1, the number of solutions of $S (a, b) = k$ is denoted by $v (k)$ and $v (k) = 4 \, \left(d_{4,1} (k) - d_{4,3} (k)\right)$. Let $k = \prod_{i = 1}^{m} p_i$ be the product of $m$ primes with $p_i \equiv 1 \mod 4$. Then all $2^m$ divisors of $k$ are $\equiv 1 \mod 4$ and $u (k) = 4 \, 2^m$ is unbounded.
\end{proof}

More generally, with the prime decomposition of $k$, the primes $p_i \equiv 1 \mod 4, q_i \equiv 3 \mod 4$ and $v (k) \ne 0$:
\begin{equation}
   k = 2^t \, \prod_{i = 1}^{m} p_i^{e_i} \, \prod_{i = 1}^{n} q_i^{2f_i} \qquad
   v (k) = 4 \, \prod_{i = 1}^{m} (e_i + 1)
   \label{Gen_v_k}
\end{equation}
Note that $v (k) = 0$ when one of the exponents of the primes $q_i$ is odd. $v (k)$ is closely related to the factorization of $k$, not only in $\Z$, but in the ring of the Gaussian integers $\Z[\omega_{1/2}] = \Z[i]$. \\

One can find in N.J.A. Sloane's OEIS \href{https://oeis.org/A004018} {A004018} and \href{https://oeis.org/A004018/b004018.txt} {Table for n = 0..10000} the number of vectors in a shell of the square lattice.  
The lowest $k$ for a given number of solutions (divided by $4$) of $S (a, b) = k$ can be found in  \href{https://oeis.org/A018782} {A018782} and \href{https://oeis.org/A018782/b018782.txt} {Table for n = 0..1432}. The table \ref{Tab_M_k_lowest} is created using OEIS A018782.\\

\begin{table}[h]
  \caption{Lowest integer $k$ for a given number of unit vectors, obtained from
           \href{https://oeis.org/A018782/b018782.txt} {A018782 Table}, \\
           \# OEIS table entry = (\# unit vectors - 8) / 4}
  \begin{tabular}{c || c | c | c | c | c | c | c | c | c | c |}
  \hline
  \# unit vectors   & 12   & 16  & 20  &  24  &  28  & 32    &    36     &   40 & \dots &          200 \\
  \hline
  $k$               &  2   & 5   & 25  &  65  &  625 &  325  &  15\,625  &  1\,105 & \dots & 5\,928\,325 \\
  \hline
  \end{tabular}
  \label{Tab_M_k_lowest}
\end{table}

\subsection{Half integer subscripts for $M_k$}

Let $k$ be the half of an odd integer, i.e. $k = (2 m + 1) / 2$. In this case the conditions for non-trival unit vectors for the integers $a, b$ reads as $S (a, b) = a^2 + b^2 = (2 m + 1) /2$. There are no solutions. If we take the half of the four bases of $M_k$, naming the new lattice as $M_k / 2$ the
condition is $S (a / 2, b / 2) = (2 m + 1) / 2$ or $S (a, b) = 2 \, (2 m + 1)$. \\

\begin{definition}
Divide the 4 bases in $M_k$ by $2$ and name the lattice $M_k / 2$:
\begin{equation}
   M_k / 2 = \{a \cdot 1 / 2 + b \cdot \omega_{1/2} / 2
             + c \cdot \omega_k / 2 + d \cdot \omega_{1/2} \, \omega_k / 2  \, \vert \, a, b, c, d \in \Z  \}   
   \label{M_k_2}
\end{equation}
The values $2 \, k = 1, 5, 13, 25, 41, 61, 85, \dots$ have to be excluded. These are the values $(s^2 - 1) / 2$ for odd $s$. They have to be excluded, because in these cases  $4 \, k - 1 = s^2$ and the lattice is only 2-dimensional.
\end{definition}

Note: in the following table the second row shows $2 \, k$ instead of $k$ in the previous tables. The values in this row are the same as in table \ref{M_k_2} for $M_k$. This is due to the fact that $v (k)$ is multiplicative, see \ref{Gen_v_k}. Because $v (2) = 1$, $v (2 \, k) = v (k)$. Exceptions for the table entries are given in the footnotes.

\newpage

\begin{table}[h]
  \caption{Lowest half integer $k$ for a given number of unit vectors, obtained from
           \href{https://oeis.org/A018782/b018782.txt} {A018782 Table}, \\
           \# OEIS table entry = (\# unit vectors - 8) / 4}
  \begin{tabular}{c || c | c | c | c | c | c | c | c | c | c |}
  \hline
  \# unit vectors   & 12$^*$   & 16$^+$  & 20$^+$  &  24  &  28  & 32    &    36     &   40 & \dots &          200 \\
  \hline
  $2 \, k$          &  9       & 13      & 169     &  65  &  625 &  325  &  15\,625  &  1\,105 & \dots & 5\,928\,325 \\
  \hline
  \end{tabular}
  \label{Tab_M_k_2_lowest}
\end{table}
{\footnotesize $ ^*$ In the OEIS table the lowest $k$ is $2$. But here the $k$ can have only odd values.
So the next odd $k$ with the same shell size is taken.} \\
{\footnotesize $ ^+$ In the OEIS table the lowest $k$ for the \# unit vectors  $= 16, 20$, i.e. for the entries $2, 3$ in OEIS, are $5, 25$. But these $M_k / 2$ build only 2-dimensional lattices. So the next lowest $k$ with the same shell size is taken, that builds a 4-dimensional lattice.}

\begin{remark}
Maybe it was not neccesary, to treat the $M_k / 2$ as own family. Assume $M_{2k}$ has $(a, b, -a, -b)$ as a non-trivial unit vector. Let $a^* = (a + b) / 2, \; b^* = (a - b) / 2$  then $(a^*, b^*, -a^*, -b^*)$ is a non-trivial unit vector in $M_k / 2$, because $a^{*2} + b^{*2} = (a^2 + b^2) / 2$. The map $a \rightarrow a^*, \, b \rightarrow b^*$ is $1:1$, therefore $M_{2k}$ and $M_k$ have the same number of unit vectors. The tables \ref{Tab_M_k_lowest} and \ref{Tab_M_k_2_lowest} have the same values in the rows (but one row represents $k$ and the other $2k$), apart from the few exceptions explained there.
\end{remark}

\subsection{The exceptional lattice $M_1$  \label{Except_M_1}}

The lattice $M_1 = L_{1/2} = \{a \cdot 1 + b \cdot \omega_{1/2} + c \cdot \omega_1 + d \cdot \omega_{1/2} \, \omega_1 \, \vert \, a, b, c, d \in \Z  \}$ has square sub-lattices and hexagonal sub-lattices. There are only $12$ default unit vectors with $H (a, c) = H (b, d) = 1$. These $12$ unit vectors van also be seen as $8$ default unit vectors with $S (a, b) = S (c, d) = 1$ and $4$ other unit vectors $(a, b, -a, -b)$. These $12$ unit vectors are given by the formulae \ref{Gen_u_k} and \ref{Gen_v_k}. In contrary to the other lattices, this lattice is self-conjugated, i.e. for a lattice point $p \in M_1$ the complex conjugated $\bar{p} \in M_1$. \\

\section{Symmetry groups of the lattices $L_k$ and $M_k$}

The symmetry group, leaving invariant the length of vectors in the lattices, is the dihedral $D_{12}$, see \cite{EnHaSuVaZs} for hexagonal sub-lattices in the $M_k, k \ne 1/2$, the dihedral $D_{8}$ for square sub-lattices in the $L_k, k \ne 1$. \\

In the exceptional case $L_{1/2}$ or $M_1$ the symmetry group is $D_{24} \sim \Z_{12} \rtimes \Z_2$ of order $24$. The $\Z_2$ in the extension corresponds to the reflection $(a, d) (b, c)$, see \cite{EnHaSuVaZs}. The $\Z_2$ also corresponds to complex conjugation. The lattice consists of the integers in the quartic extension $x^4 - x^2 + 1$, defining the four primitive $12^{th}$ roots of unity. The Kleinian 4-group is the Galois group of this quartic polynomial with the $4$ field automorphisms $x \rightarrow x^1, x \rightarrow x^5, x \rightarrow x^7, x \rightarrow x^{11}$. Complex conjugation is the field automorphism $x \rightarrow x^{11} = x^{-1}$. \\

Geometrically the $D_{24}$ is generated by a primitive rotation of a regular 12-gon and a reflection on a line through the midpoints of $2$ opposite sides, which represents the reflection $(a, d) (b, c)$. Instead of this reflection another reflection on a line through $2$ opp0site vertices, the real axis could be taken as generator, representing complex conjugation.

\newpage

\bibliographystyle{amsplain}

\end{document}